\documentclass[11pt,pdftex, reqno]{amsart}

\usepackage{amsmath, amsthm}
\usepackage{eucal}

\usepackage{palatino}
\usepackage{euler}

\usepackage{amssymb}
\usepackage{amscd}
\usepackage{latexsym}
\usepackage{epsfig}
\usepackage{graphicx, xcolor}
\usepackage{amsfonts}
\usepackage{psfrag}
\usepackage{caption}

\usepackage{fullpage}
\usepackage{draftcopy}
\usepackage{tikz}

\usepackage{setspace}

\usepackage{pifont}

\usepackage{verbatim}

\usepackage{dsfont}

\usepackage{url}
\usepackage{cite}

\usepackage{dsfont}

\usepackage[margin=1in]{geometry}

\input xy
\xyoption{all}
\UseComputerModernTips

\numberwithin{equation}{section}
\numberwithin{figure}{section}

\newtheorem{theorem}{Theorem}[section]

\newtheorem{lemma}[theorem]{Lemma}
\newtheorem{proposition}[theorem]{Proposition}
\newtheorem{corollary}[theorem]{Corollary}

\theoremstyle{definition}

\newcommand{\C}{{\mathbb{C}}}

\renewcommand{\b}{\mathfrak{b}}

\newcommand{\fl}{\mathcal{F}\ell}
\newcommand{\spn}{\mathfrak{sp}_n}

\newcommand{\defi}{\textbf}
\newcommand{\dom}{\backslash}
\newcommand{\pf}{\text{pf}}
\newcommand{\init}{\mbox{init }}

\definecolor{light-gray}{gray}{0.95}

\begin{document}

\title{The Orbits of the Symplectic Group on the Flag Manifold}

\author{Anna Bertiger}



\date{\today}


\begin{abstract}
 We examine the orbits of the (complex) symplectic group, $Sp_n$, on the flag manifold, $\fl(\C^{2n})$, in a very concrete way.  We use two approaches: we Gr\"obner degenerate the orbits to unions of Schubert varieties (for a equations of a particular union of Schubert varieties see \cite{NWunions}) and we find a subset $\mathcal{A}$ of the orbit closures containing the basic elements of the poset of orbit closures under containment, which represent the geometric and combinatorial building blocks for the orbit closures.  
\end{abstract}

\maketitle

\section{Introduction}

The problem of understanding the $Sp_n$ orbits on $GL_n/B$ is a particular instance of the more general problem of studying the orbits of $K$ on $G/B$, where $K$ is the set of fixed points of an involution $\theta$ on $G$.  This problem is of interest in geometric representation theory.  We examine the particular case of $G=GL_n\C$ and $K=Sp_n$ because it lends itself to using tools of Frobenius splitting, particularly in the proof of Theorem \ref{thm: whichSchuberts}.  

Using different approaches, we present two main results about the action of the symplectic group on the flag manifold.   We give precise statements as Theorems \ref{thm:basicElts} and \ref{thm: whichSchuberts}, but 
describe the results informally here.  

We first examine the poset of $Sp_n$ orbit closures under containment, reducing the problem by finding basic elements of this poset.  Basic elements of a poset are a set of elements from which all other elements can be found by taking the meet of a subset of the basic elements.  They take on the role of a basis in a vector space in the sense of encapsulating the structure of the object.  In the case of a poset of varieties under containment with reduced intersections, as in the case of $Sp_n$ orbit closures on the flag manifold, the defining equations of the basic varieties would be sufficient to easily produce equations for all other varieties.  
\begin{theorem}[informal version of Theorem \ref{thm:basicElts}]
There is a combinatorially defined set $\mathcal{A}$ which contains the basic elements of the partially ordered set of orbit closures under containment.  All other orbit closures can be found by intersecting the basic orbit closures.  
\end{theorem}

Secondly, we Gr\"obner degenerate the orbit closures of $Sp_n$ on $\fl(\C^{2n})$.  For each orbit closure, we describe the variety corresponding to the initial form of the ideal defining the orbit closure.  
\begin{theorem}[informal version of Theorem \ref{thm: whichSchuberts}]
The orbits of $Sp_n$ on $GL_n\C/B$ each Gr\"obner degenerate, or deform along a flat family, to a union of Schubert varieties given by a combinatorial operation on the fixed point free involution indexing the orbit.  
\end{theorem}
Algebraically, this is a rare case in which we can describe the initial ideal but not the Gr\"obner basis for the original ideal.  

One reasonable starting point for the history of this problem is a paper of Richardson and Springer \cite{RichardsonSpringer1990}.  In that paper and its successor \cite{RichardsonSpringer1993}, Richardson and Springer give combinatorial criteria for containment of orbit closures $\overline{BgK}$ as $g$ ranges over $G$.  Later, Springer \cite{Springer2003} produced invariants for a reduced decomposition of $B$ acting on $G/K$ and used the reduced decomposition of $G/K$ to look at the decomposition of the wonderful compactification of $G/K$.  Brion and Helminck \cite{BrionHelminck} studied the geometry of $K$ orbits on $G/P$ for more general parabolics $P$.  He and Thomsen \cite{HeThomsen} used results of Knutson \cite{KnutsonFrob} to investigate the Frobenius splitting of these orbits.  Wyser \cite{Wyser} gave a correspondence between $K$ orbits on $G/B$ and Richardson varieties for certain $G$ and $K$ including $G=GL_n\C$ and $K=Sp_n$.  Hultman \cite{Hultman} gave combinatorial criteria for rational smoothness of particular symmetric orbit closures.  

On the purely combinatorial side, Ignatyev \cite{Ignatyev} showed that the order given by containment of $Sp_n$ orbit closures in the flag manifold corresponds to the opposite Bruhat order on fixed point free involutions on $\{1, \ldots, 2n\}$.  Happily, Incitti \cite{Incitti} had already studied the Bruhat order on involutions in the symmetric group.  

\subsection{Background and Notation}

Let $J$ be the $2n \times 2n$ block diagonal matrix with diagonal blocks $J_2=\left(\begin{smallmatrix}0 & 1 \\-1 & 0\end{smallmatrix}\right)$.  The \defi{symplectic group}, $Sp_n$, is the subgroup of $2n \times 2n$ invertible matrices with complex entries that preserve the symplectic form $J$, i.e. $Sp_n=\{M \in GL_{2n}(\C) : MJM^T=J\}$.  

The \defi{flag manifold}, $\fl(\C^n)$, is composed of nested chains of vector subspaces $\{V_0 \subsetneq V_1 \subsetneq \cdots \subsetneq V_n\}$ of $\C^n$, where each subspace $V_i$ is of dimension $i$.  We can also think of $\fl(\C^n)$ as $B\dom GL_n\C$, where $B=B_-$ is the subgroup of invertible lower triangular $n \times n$ matrices.  From this point of view, $Sp_n$ acts naturally on $\fl(\C^{2n})$ by right multiplication by an inverse on the cosets of $B \dom GL_{2n}\C$.  For technical ease, we shall study the $B \times Sp_n$ orbits on $GL_{2n}\C$, rather than the $Sp_n$ orbits on $\fl(\C^{2n})$.  

We begin with some background on the flag manifold and the symmetric group.  Let $B_+$ be the subgroup of upper triangular invertible matrices and let $M_\pi$ be the permutation matrix for a permutation $\pi$.  The \defi{Schubert variety} $X_\pi$ in $\fl(\C^n)$ is $B \dom \overline{B M_\pi B_+}$.   The corresponding \defi{matrix Schubert variety} $\overline{X}_\pi$ is the closure  $\overline{B_-M_\pi B_+}$ in the space of all $n \times n$ matrices.    Matrix schubert varieties were discovered by Fulton in his study of $\fl(\C^n)$ \cite{Fulton1992}.   Knutson and Miller \cite{KnutsonMiller2005} gave combinatorial results describing the algebro-geometric properties of matrix Schubert varieties using a Gr\"obner basis approach.  We summarize below the portion of those results applicable to our study here.  

The \defi{Rothe diagram} of a permutation  is the cells that remain in the permutation matrix after crossing out the cells weakly below and to the right of each $1$.  For example, the Rothe diagrams for $2143$ and $15432$ are given in Figure \ref{fig:essSets}.  The \defi{essential boxes} \cite{Fulton1992} of a permutation are those boxes in the Rothe diagram that do not have any boxes of the diagram immediately south or east of them.  These are the boxes are marked with $e$ in Figure \ref{fig:essSets}.  

\begin{figure}
\begin{center}
\begin{tikzpicture}[x=.75cm, y=.75cm] 
\draw (.5,2.5) node{$1$};
\draw [dashed, thick] (.5,0)--(.5,2.5)--(4,2.5);
\draw (1.5,3.5) node{$1$};
\draw [dashed, thick] (1.5,0)--(1.5,3.5)--(4,3.5);
\draw (3.5,1.5) node{$1$};
\draw [dashed, thick] (3.5,0)--(3.5,1.5)--(4,1.5);
\draw (2.5,.5) node{$1$};
\draw [dashed, thick] (2.5,0)--(2.5,.5)--(4,.5);
\draw (0,3)--(1,3)--(1,4)--(0,4)--(0,3);
\draw (.5,3.5) node{$e$};
\draw (2,1)--(3,1)--(3,2)--(2,2)--(2,1);
\draw (2.5,1.5) node{$e$};
\draw (6.5,4.5) node{$1$};
\draw [dashed, thick] (6.5,0)--(6.5,4.5)--(11,4.5);
\draw (10.5,3.5) node{$1$};
\draw [dashed, thick] (10.5,0)--(10.5,3.5)--(11,3.5);
\draw (9.5,2.5) node{$1$};
\draw [dashed, thick] (9.5,0)--(9.5,2.5)--(11,2.5);
\draw (8.5,1.5) node{$1$};
\draw [dashed, thick] (8.5,0)--(8.5,1.5)--(11,1.5);
\draw (7.5,.5) node{$1$};
\draw [dashed, thick] (7.5,0)--(7.5,.5)--(11,.5);
\draw (7,1)--(7,4)--(10,4)--(10,3)--(7,3);
\draw (7,2)--(9,2)--(9,4);
\draw (7,1)--(8,1)--(8,4);
\draw (7.5,1.5) node{$e$};
\draw (8.5,2.5) node{$e$};
\draw (9.5,3.5) node{$e$};
\end{tikzpicture}
\caption{The Rothe diagrams and essential sets of $2143$ (left) and $15432$ (right).}
\label{fig:essSets}
\end{center}
\end{figure}
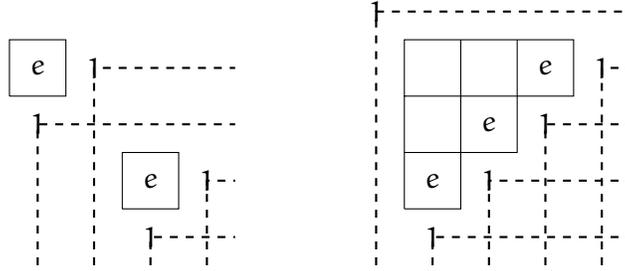

The \defi{rank matrix}  of $\pi$, $r(\pi)$, has entries $r_{ij}(\pi)=\#\{k \le i : \pi(k) \le j\}$.  For example, the rank matrix of $13425$ is given by:

\[
\left(\begin{array}{ccccc}1 & 1 & 1 & 1 & 1 \\1 & 1 & 2 & 2 & 2 \\1 & 1 & 2 & 3 & 3 \\1 & 2 & 3 & 4 & 4 \\1 & 2 & 3 & 4 & 5\end{array}\right).  
\] 
We will impose the \defi{(strong) Bruhat order} on elements of $S_n$, where $\pi$ covers $\rho$ if $\pi=\rho t_{ij}$ and $l(\pi)=l(\rho)+1$.  Here $t_{ij}$ is the transposition that switches $i$ and $j$.  For example $2134$ covers $1234$ but $3214$ does not.  Equivalently, the Bruhat order is given by entry-wise comparison of the rank matrices of the two permutations, i.e. $\pi \ge \rho$ if and only if $r_{ij}(\pi) \ge r_{ij}(\rho)$ for all $i$ and $j$.  This corresponds to the ordering of Schubert varieties by reverse containment.  In this paper the order on $S_n$ that will be most useful is the \defi{opposite Bruhat order}, reversing the relations in Bruhat order, and ordering the Schubert varieties by containment.   

The \defi{basic elements} \cite{KnutsonFrob} of a partially ordered set are the elements from which all other elements can be found by taking the unique greatest lower bound of larger basic elements in the poset.  For example, the basic elements in $S_n$ under the opposite Bruhat order are those permutations with only one essential box.  The collection $\{X_\pi: \pi \text{ is a basic element of } S_n\}$ are the basic elements in the poset of Schubert varieties ordered by containment.   Each Schubert variety $X_\pi$ is the intersection over the essential boxes of $\pi$ of the Schubert varieties with only one essential box and the same essential rank condition as the corresponding essential rank condition of $\pi$.  If the subsets in the poset are compatibly split, as they are in the Schubert case, then finding the equations for the basic elements is sufficient to find equations for all elements.  This is because intersecting varieties corresponds to summing their ideals, i.e. concatenating their lists of ideal generators.  

\begin{theorem}[\cite{Fulton1992}]\label{thm:Fulton1992}
Matrix Schubert varieties have corresponding radical ideal $I(\overline{X}_\pi)=I_\pi$ given by determinants representing conditions given in the rank matrix $r(\pi)$, that is, the $(r_{ij}(\pi)+1) \times (r_{ij}(\pi)+1)$ minors of each northwest $i \times j$ submatrix of a matrix of variables.  In fact, it is sufficient to impose only those rank conditions $r_{ij}(\pi)$ such that $(i,j)$ is an essential box for $\pi$.  
\end{theorem}

Henceforth, we shall call the generators for $I_\pi$ given by the essential rank conditions the \defi{Fulton generators} for $I_\pi$.  

Fix a total ordering on the monomials of a polynomial ring, perhaps by putting incommensurable weightings on the variables and comparing the total weight of each monomial.  If $1 < m$ for all monomials $m$ and if $m < n$ implies $pm < pn$ for all monomials $p$, we will call such an ordering of the monomials a \defi{term order}.  The largest monomial appearing in any polynomial $f$ is the \defi{initial term} of $f$, denoted $\init f$. For example, in polynomials in one variable the initial term of a polynomial is its largest degree term.   

The \defi{initial ideal} of an ideal $I$ is $\init I := \langle \init f: f \in I \rangle$.  A subset $\{f_1, \ldots, f_r\}$ of $I$ is a \defi{Gr\"obner basis} for $I$ if $\init I = \langle \init f_1, \ldots , \init f_r \rangle$.  Notice that a Gr\"obner basis for $I$ is also a generating set for $I$.  If $I$ is the defining ideal for a scheme, the scheme corresponding to $\init I$ is geometrically related to the one corresponding to $I$; the initial scheme can be found from the original scheme by deforming along a flat family.  

In the study of Schubert varieties we will be particularly interested in antidiagonal term orders.  The \defi{antidiagonal} of a matrix is the diagonal series of cells in the matrix running from the most northeast to the most southwest cell.  The \defi{antidiagonal term} (or \defi{antidiagonal}) of a determinant is the product of the entries in the antidiagonal.  For example, the antidiagonal of $\left(\begin{smallmatrix}a & b \\c & d\end{smallmatrix}\right)$ is the cells occupied by $b$ and $c$, and correspondingly, in the determinant $ad-bc$ the antidiagonal term is $bc$.  Term orders that select antidiagonal terms as the initial term of a determinant, called \defi{antidiagonal term orders}, have proven especially useful for ideals of matrix Schubert varieties.   There are several possible implementations of an antidiagonal term order on an $n \times n$ matrix of variables, for example weighting the top right entry the highest and rastering across the matrix right to left and then top to bottom, reducing weights as entries are passed.  One use of antidiagonal term orders is a result of Knutson and Miller's:

\begin{theorem}[\cite{KnutsonMiller2005}]
The Fulton generators for the ideal $I_\pi$ form a Gr\"obner basis under any antidiagonal term order.  
\end{theorem}


A \defi{fixed-point-free involution} is an element $\iota \in S_{2n}$ such that $\iota^2=\text{identity}$ and $\iota(i) \ne i$ for $1 \le i \le 2n$.  Begin with a $1 \times 2n$ array of dots which we will call \defi{outlets}.  A \defi{wiring diagram} for a fixed-point-free involution is the figure formed by connecting the $i^{th}$ and $\iota(i)^{th}$ outlets with an arc, or wire, run over the array of dots.  For example, the wiring diagram for $43217856$ is given in Figure \ref{fig:wiringDiag}.  

\begin{figure}
\begin{center}
\begin{tikzpicture}
\draw plot [smooth, tension=1] coordinates {(1,0) (2.5,1) (4,0)};
\draw plot [smooth, tension=1] coordinates {(2,0) (2.5,.5) (3,0)};
\draw plot [smooth, tension=1] coordinates {(5,0) (6,.5) (7,0)};
\draw plot [smooth, tension=1] coordinates {(6,0) (7,.5) (8,0)};
\end{tikzpicture}
\end{center}
\caption{The wiring diagram for $43217856$.}
\label{fig:wiringDiag}
\end{figure}
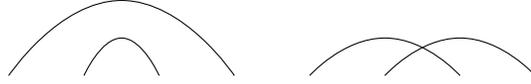

The fixed point free involution $2143 \cdots (2n)(2n-1)$ (written as a permutation in one line notation) will be particularly important and we will denote it $\overline{J}_n$.  The privileged position of $\overline{J}_n$ stems both from the fact that it is the unique shortest fixed point free involution and from the fact that it is the combinatorial shadow of the symplectic form $J$.

Our key tool to begin understanding the $Sp_n$-orbit closures on $\fl(\C^{2n})$ is:
\begin{theorem}[\cite{RichardsonSpringer1990}]
$Sp_n$ orbits on $\fl(\C^{2n})$ correspond to fixed-point-free involutions.  
\end{theorem}

\emph{Proof}  See \cite{RichardsonSpringer1993} example $5.1(4)$ and \cite{RichardsonSpringer1990} Section $10$.  We map from 
\[
\{\text{orbits of }B_- \times Sp_n \text{ on } 2n \times 2n \text{ full rank matrices}\} \to  \{\text{orbits of } B_- \text{ acting by }b \cdot M= bMb^T\}
\] 
via the map on matrices $M \mapsto MJM^T$. \qed

\noindent We shall denote by $Y_\iota$ the orbit corresponding to the fixed-point-free involution $\iota$.  

\subsection{Basic Elements of the Poset of $Y_\iota$}  This section describes the result of finding a subset of the partially ordered set of orbit closures that contains the basic elements.  Endow the set of fixed point free involutions with the partial order inherited from the opposite Bruhat order on the symmetric group.  We will define a set $\mathcal{A}$ that contains the basic elements of this poset.  We begin with two definitions:  The \defi{symplectic diagram} for a fixed point free involution is formed by taking the usual Rothe diagram for the permutation $\iota$ and intersecting it with the strict upper triangle.  The \defi{symplectic essential boxes} for a fixed-point-free involution $\iota$ are  the boxes in the symplectic diagram with no symplectic diagram cells immediately south or immediately east of them.  Essentially, this the definition used in Fulton's Theorem restricted to the strict upper triangle of the Rothe diagram. The boxes in the symplectic diagram for $216543$ are shown in Figure \ref{fig:symplDiag} with the symplectic essential box marked with an $e$.  
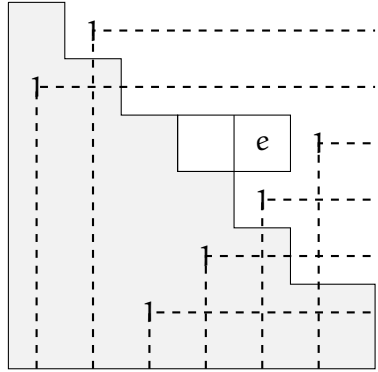
\begin{figure}
\begin{center}
\begin{tikzpicture}[x=.75cm, y=.75cm]
\draw [fill=light-gray](-.5,6.5)--(.5,6.5)--(.5,5.5)--(1.5,5.5)--(1.5,4.5)--(2.5,4.5)--(2.5,3.5)--(3.5,3.5)--(3.5,2.5)--(4.5,2.5)--(4.5,1.5)--(6,1.5)--(6,0)--(-.5,0)--(-.5,6.5);
\draw (0,5) node{$1$};
\draw [dashed, thick] (0,0) --(0,5)--(6,5);
\draw (1,6) node{$1$};
\draw [dashed, thick] (1,0) --(1,6)--(6,6);
\draw (2,1) node{$1$};
\draw [dashed, thick] (2,0) --(2,1)--(6,1);
\draw (3,2) node{$1$};
\draw [dashed, thick] (3,0) --(3,2)--(6,2);
\draw (4,3) node{$1$};
\draw [dashed, thick] (4,0) --(4,3)--(6,3);
\draw (5,4) node{$1$};
\draw [dashed, thick] (5,0) --(5,4)--(6,4);
\draw (2.5,3.5)--(4.5,3.5)--(4.5,4.5)--(2.5,4.5)--(2.5,3.5);
\draw (4,4) node{$e$};
\draw (3.5,3.5)--(3.5,4.5);
\end{tikzpicture}
\end{center}
\caption{The symplectic diagram for $216543$ is the unshaded portion of the diagram.}
\label{fig:symplDiag}
\end{figure}

Define $\mathcal{A}=\mathcal{A}_e \cup \mathcal{A}_o$ to be a subset of the fixed-point-free involutions such that $\{Y_\iota:\iota \in \mathcal{A}\}$ contains the basic elements for the poset of $\{Y_\iota\}$.  $\mathcal{A}_e$ is the set of all fixed-point-free involutions with exactly one symplectic essential box which corresponds to an even rank condition $2r$.  $\mathcal{A}_o$ is the set of all fixed-point-free involutions with exactly two symplectic essential boxes one at $(p,p+1)$ with even rank condition $2r$ and one in row $p+1$ with odd rank condition $2r+1$.  

\begin{theorem}\label{thm:basicElts}
$\mathcal{A}$ contains the basic elements of the set of fixed point free involutions.  Hence $\mathcal{A}$ gives a corresponding subset of the poset of orbit closures of $Sp_n$ on $\fl(\C^{2n})$ that contains the basic elements.  \end{theorem}

Despite being considerably smaller than the entire set of fixed point free involutions, $\mathcal{A}$ contains more than just the basic elements.  In the case of $2n=6$, for example, $351624 \in \mathcal{A}$, but it is the greatest lower bound of $341265$ and $215634$.    

\subsection{Degenerating the $Y_\iota$}

This section states the result of Gr\"obner degenerating an orbit closure.  We will show that, for a suitable term order, an orbit closure $Y_\iota$ Gr\"obner degenerates to the union of the Schubert varieties for ``pair permutations" for $\iota$.  A \defi{pair permutation}  for a fixed-point-free involution $\iota$ is a permutation found by:
\begin{itemize}
\item Write the wiring diagram for $\iota$ using outlets across the top of an array 
\item Write the wiring diagram for $\overline{J}_n$ using outlets along the left side of the array 
\item Connect the half loops formed by the wiring diagrams for $\iota$ and $\overline{J}_n$ into circles with more wires, so as the minimize the number of potential wire crossings
\item Read the pair permutation by treating the newly added wires as a wiring diagram for a permutation reading from top to left.  If a single wire goes from the $i^{th}$ from left outlet on the top to the $j^{th}$ from the top outlet on the left then the pair permutation will take $i$ to $j$.  
\end{itemize}
Denote the set of all possible pair permutations for $\iota$ by $P(\iota)$.  For example, $P(4321)=\{1342, 3124\}$.  The diagrams used in the construction of $P(4321)$ are shown in Figure \ref{fig:pairPerms}.  

The correct term order is given by weighting the variables by weighting the variables in column $j$ of the matrix of variables by $t^{\lceil j/2\rceil-1}$.  The initial form of an equation is the result of taking the limit as $t$ goes to $0$ of each term of the equation times its total weight.  This is a relaxation of the notion of a term order in that it allows ties in the order on monomials.  We shall thus consider the initial form of a polynomial as the sum of the largest terms in the polynomial.

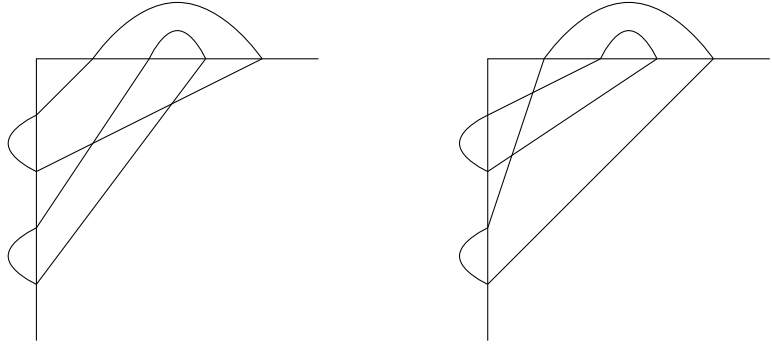
\begin{figure}
\begin{tikzpicture}[x=.75cm, y=.75cm]
\draw plot [smooth, tension=1] coordinates {(2,5) (3.5,6) (5,5)};
\draw plot [smooth, tension=1] coordinates {(1,3) (.5,3.5) (1,4)}; 
\draw plot [smooth, tension=1] coordinates {(3,5) (3.5,5.5) (4,5)};
\draw plot [smooth, tension=1] coordinates {(1,1) (.5,1.5) (1,2)};
\draw (1,4)--(2,5);
\draw (5,5) --(1,3);
\draw (1,2)--(3,5);
\draw (1,1)--(4,5);
\draw (1,0)--(1,5)--(6,5);
\draw plot [smooth, tension=1] coordinates {(10,5) (11.5,6) (13,5)};
\draw plot [smooth, tension=1] coordinates {(9,3) (8.5,3.5) (9,4)}; 
\draw plot [smooth, tension=1] coordinates {(11,5) (11.5,5.5) (12,5)};
\draw plot [smooth, tension=1] coordinates {(9,1) (8.5,1.5) (9,2)};
\draw (9,4)--(11,5);
\draw (12,5) --(9,3);
\draw (9,2)--(10,5);
\draw (9,1)--(13,5);
\draw (9,0)--(9,5)--(14,5);
\end{tikzpicture}
\caption{The pair permutations for $4321$ are $1342$ and $3124$.}
\label{fig:pairPerms}
\end{figure}

\begin{theorem}\label{thm: whichSchuberts}
$Y_\iota$ degenerates to $\cup_{\pi \in P(\iota)} X_\pi$ under the relaxed notion of a term order described above.  That is, for the weighting of variables given above, 
\[
\init I(Y_\iota)=\init\cap_{\pi \in P(\iota)}I_\pi.
\]
\end{theorem}

\subsection{Structure of Paper}  We begin with a section of background on fixed point free involutions and their specific combinatorics relating to the action of $Sp_n$ on $\fl(\C^{2n})$.  This is followed by a section containing the proof of Theorem \ref{thm:basicElts} and then a section containing the proof of Theorem \ref{thm: whichSchuberts}.  We conclude with a section of data on progress toward finding generators of $I(Y_\iota)$ using Theorems \ref{thm:basicElts} and \ref{thm: whichSchuberts}.  

\subsection{Acknowledgements}This project constitutes part of my PhD thesis completed at Cornell University.  I'd like to thank my advisor, Allen Knutson, for his help and support.  I'd also like to thank Eric Katz, Kevin Purbhoo and Kaisa Taipale for their very helpful comments on drafts of this paper.  

\section{Background on Fixed Point Free Involutions}

In the order on fixed point free involutions inherited from  the opposite Bruhat order on $S_n$, the covering relation is given by \emph{conjugating by} any transposition resulting in a decrease in permutation length by $2$.  

\begin{proposition}[\cite{Incitti}]
The covering relations in the Bruhat order on fixed-point-free involutions (thought of as wiring diagrams) can be described by switching the plugs in two outlets such that the length of the new involution increases by $2$.  
\end{proposition}

%
For example, if we switch the $1$ and the $4$ in $214365$ we must also switch the $2$ and $3$ in order to still have a fixed point free involution; this corresponds to conjugating by $t_{23}$.  This is the same as switching the plugs corresponding to the ends of the wire connecting $1$ and $2$ and the wire connecting $3$ and $4$.  In the opposite Bruhat order on fixed point free involutions order, $341265$ is covered by $214365$.  The wiring diagrams for $341265$ and $214365$ are shown in Figure \ref{fig:wiringCovers}.  We can, of course, switch wires in non-adjacent outlets $i$ and $j$ which correspond to conjugating by $t_{ij}$.  If conjugating by $t_{ij}$ changes the length of a transposition by $2$ it still corresponds to a cover in this order, as it is the order inherited from the strong Bruhat order on $S_n$.

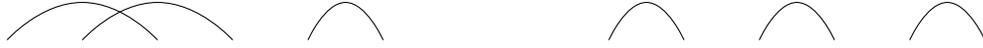
\begin{figure}
\begin{center}
\begin{tikzpicture}
\draw plot [smooth, tension=1] coordinates {(1,0) (2,.5) (3,0)};
\draw plot [smooth, tension=1] coordinates {(2,0) (3,.5) (4,0)};
\draw plot [smooth, tension=1] coordinates {(5,0) (5.5,.5) (6,0)};
\draw plot [smooth, tension=1] coordinates {(9,0) (9.5,.5) (10,0)};
\draw plot [smooth, tension=1] coordinates {(11,0) (11.5,.5) (12,0)};
\draw plot [smooth, tension=1] coordinates {(13,0) (13.5,.5) (14,0)};
\end{tikzpicture}
\end{center}
\caption{$341265$ is covered by $214365$.}
\label{fig:wiringCovers}
\end{figure}

Let $\iota$ be an involution of $1, \ldots, 2k$ and $\iota'$ be an involution of $1, \ldots, 2m$.  We shall define the involution $\iota \oplus \iota'$ of $1, \ldots , 2k+2m$ by
\begin{displaymath}
   (\iota \oplus \iota')(i) = \left\{
     \begin{array}{lr}
       \iota(i) & i \in \{1, \ldots, 2k\}\\
       \iota'(i-2k)+2k &  i \in \{2k+1, \ldots, 2k+2m\}
     \end{array}
   \right.
\end{displaymath}
Taking the direct sum of $\iota$ and $\iota'$ is equivalent to drawing the wiring diagrams for $\iota$ and $\iota'$ next to each other and considering the result as the wiring diagram for one fixed point free involution.  We will say that $\iota$ is a \defi{countryside involution} if $\iota$ is of the form $\overline{J}_{k} \oplus 3412 \oplus \overline{J}_{n-k-2}$ and that $\iota$ is a \defi{rainbow involution} if $\iota$ is of the form $\overline{J}_{k} \oplus 4321 \oplus \overline{J}_{n-k-2}$.  
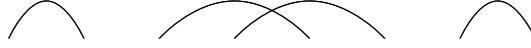
\begin{figure}
\begin{center}
\begin{tikzpicture}
\draw plot [smooth, tension=1] coordinates {(1,0) (1.5,.5) (2,0)};
\draw plot [smooth, tension=1] coordinates {(3,0) (4,.5) (5,0)};
\draw plot [smooth, tension=1] coordinates {(4,0) (5,.5) (6,0)};
\draw plot [smooth, tension=1] coordinates {(7,0) (7.5,.5) (8,0)};
\draw plot [smooth, tension=1] coordinates {(1,0) (1.5,.5) (2,0)};
\draw plot [smooth, tension=1] coordinates {(3,0) (4,.5) (5,0)};
\draw plot [smooth, tension=1] coordinates {(4,0) (5,.5) (6,0)};
\draw plot [smooth, tension=1] coordinates {(7,0) (7.5,.5) (8,0)};
\end{tikzpicture}
\end{center}
\caption{$21563487$ is a countryside involution.}
\label{fig:countryside}
\end{figure}
\begin{figure}
\begin{center}
\begin{tikzpicture}
\draw plot [smooth, tension=1] coordinates {(1,0) (1.5,.5) (2,0)};
\draw plot [smooth, tension=1] coordinates {(3,0) (3.5,.5) (4,0)};
\draw plot [smooth, tension=1] coordinates {(5,0) (6.5,1) (8,0)};
\draw plot [smooth, tension=1] coordinates {(6,0) (6.5,.5) (7,0)};
\draw plot [smooth, tension=1] coordinates {(9,0) (9.5,.5) (10,0)};
\end{tikzpicture}
\end{center}
\caption{$21438765109$ is a rainbow involution.}
\label{fig:rainbow}
\end{figure}
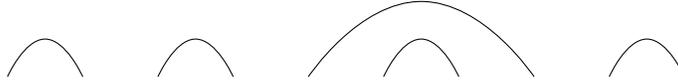
We will say that one involution $\iota$ \defi{contains} another involution $\iota'$ if by removing enough wires from the wiring diagram of $\iota$ we are left with the wiring diagram for $\iota'$.  For example, $532614$ contains both a rainbow involution and a countryside involution as shown in Figure \ref{fig:containsEx}.  

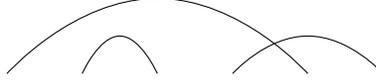
\begin{figure}
\begin{center}
\begin{tikzpicture}
\draw plot [smooth, tension=1] coordinates {(1,0) (3,1) (5,0)};
\draw plot [smooth, tension=1] coordinates {(2,0) (2.5,.5) (3,0)};
\draw plot [smooth, tension=1] coordinates {(4,0) (5,.5) (6,0)};
\end{tikzpicture}
\end{center}
\caption{$532614$ contains both a rainbow involution and a countryside involution. }
\label{fig:containsEx}
\end{figure}

%
%
%
%

Notice that we can write an involution by writing a history of outlet switches corresponding to switching the wires plugged into a pair of outlets.  This is the analog of a reduced word in $S_n$, and is similarly not unique.  In fact, switching wires in outlets $i$ and $i+1$ corresponds to conjugating by the simple reflection $s_i$.  This leads us to a formula for the length of a fixed point free involution in terms of the structure of its wiring diagram:

\begin{lemma}
Let $c$ be the number of countryside involutions contained in $\iota$ and $r$ be the number of rainbow involutions contained in $\iota$.  Then the permutation length of a fixed-point-free involution $\iota$ of $1, \ldots, 2n$ is given by $n+2c+4r$.  
\end{lemma}

\begin{proof}
Proceed by induction on the $l(\iota)$.  For the base case examine the fixed-point-free involution $\overline{J}_n$, which has a reduced word given by $s_1 s_3 \cdots s_{2n-1}$ and length $n=n+2c+4r$, since $c$ and $r$ are both $0$.  Now assume that the Lemma is true for all fixed-point-free involutions of length at most $l$ for a fixed $l \ge n$.  Fix a fixed-point-free involution $\iota$ of length $l+1$ and examine its wiring diagram.  There are at least two adjacent outlets whose wires can be interchanged to produce a shorter permutation $\iota'$.  Pick the rightmost pair of adjacent outlets at locations $i$ and $i+1$ such that $\iota(i) \ne i+1$ and $\iota(i)> \iota(i+1)$.   Conjugating $\iota$ by $s_i$, i.e. switching the wires in the $i$ and $i+1$ outlets, produces an involution of length $l$, and hence one to which our inductive hypothesis applies.  This switch also changes the two wires involved in one of two ways: it turns them from a countryside involution to the involution $\overline{J}_2$ or from a rainbow involution to a countryside involution.  Each of these changes reduces the formula $n+2c+4r$ by two.  
\end{proof} 

\begin{corollary}
The covering relations for fixed point free involutions are given by switching the wires in two outlets where exactly one pair of wires changes from a pair of side by side loops to a countryside involution or from a countryside involution to a rainbow involution and all other pairs of wires stay in the same relative orientation.  
\end{corollary}

\section{Proof of Theorem \ref{thm:basicElts}}

In this section we prove Theorem \ref{thm:basicElts} which gives a description of a set $\mathcal{A}$ containing the basic elements in the poset of fixed point free involutions under the opposite Bruhat order.  Recall that $\mathcal{A}=\mathcal{A}_o \cup \mathcal{A}_o$, where $\mathcal{A}_e$ is the set of fixed point free involutions with one symplectic essential box with rank condition $2r$ and $\mathcal{A}_o$ is those with one symplectic essential box with rank condition $2r$ at $(p,p+1)$ and a second symplectic essential box with rank condition $2r+1$ in row $p+1$.  We begin by constructing the elements in the set $\mathcal{A}$ and then show that any other fixed point free involution can be written at the greatest lower bound of elements in $\mathcal{A}$ to prove Theorem \ref{thm:basicElts}.  

There are two cases of possible symplectic rank conditions in $\mathcal{A}$, depending on the parity of the rank condition we wish to impose, i.e. they correspond to $\mathcal{A}_e$ and $\mathcal{A}_o$.   Each of these cases has two sub-cases depending on the relative parity of the coordinates $(i,j)$ of the prospective essential box's location.  

We begin by showing how to impose one even symplectic essential rank condition, i.e. by constructing the fixed point free involutions in $\mathcal{A}_e$.  To obtain a symplectic essential rank condition $2r$ at $(i,j)$ where $i>2r$ and $j>i$ and no other symplectic essential boxes, we need two cases depending on whether $i-j$ is even or odd:  

If $i-j$ is even, the permutation $\overline{J}_r \oplus  \iota_e(i-2r,j-2r) \oplus \overline{J}_{n-r-(i+j)/2}$ has exactly one symplectic essential box at $(i,j)$ with rank condition $2r$.  $\iota_e(a,b)$ is described in Figure \ref{fig:iota_e}.  For example, $216543=\overline{J}_2\oplus\iota_e(1,2)$ has one symplectic essential box at $(3,5)$ with rank condition $2$ associated to it.  If $i-j$ is odd, the permutation $\overline{J}_r \oplus \iota_e'(i-2r,j-2r) \oplus \overline{J}_{n-r-(i+j+1)/2}$ has exactly one symplectic essential box at $(i,j)$ with rank condition $2r$.  $\iota_e'(a,b)$ is given in Figure \ref{fig:iota_e'}.  For example, $\iota_e'(5,1)=73254816$ has one symplectic essential box at $(1,6)$ with rank condition $0$ associated to it.  

\begin{figure}
\begin{center}
\begin{tikzpicture}
\draw plot [smooth, tension=1] coordinates {(1,0) (5,1) (9,0)};
\draw plot [smooth, tension=1] coordinates {(2,0) (6,1) (10,0)};
\draw plot [smooth, tension=1] coordinates {(3,0) (7,1) (11,0)};
\draw plot [smooth, tension=1] coordinates {(4,0) (4.5,.5) (5,0)};
\draw plot [smooth, tension=1] coordinates {(7,0) (7.5,.5) (8,0)};
\node [below] at (2,0) {$b$ strands};
\node at (6,0) {$\cdots$};
\node [below] at (6,0) {a copy of $\overline{J}_{a-b}$};
\end{tikzpicture}
\end{center}
\caption{$\iota_e(a,b)$}
\label{fig:iota_e}
\end{figure}
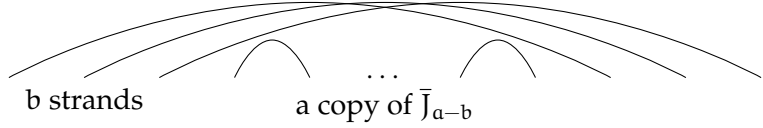

\begin{figure}
\begin{center}
\begin{tikzpicture}
\draw plot [smooth, tension=1] coordinates {(1,0) (5.5,1) (10,0)};
\draw plot [smooth, tension=1] coordinates {(2,0) (6.5,1) (11,0)};
\draw plot [smooth, tension=1] coordinates {(3,0) (7.5,1) (12,0)};
\draw plot [smooth, tension=1] coordinates {(4,0) (4.5,.5) (5,0)};
\draw plot [smooth, tension=1] coordinates {(7,0) (7.5,.5) (8,0)};
\draw plot [smooth, tension=1] coordinates {(9,0) (11,1) (13,0)};
\node [below] at (2,0) {$b$ strands};
\node at (6,0) {$\cdots$};
\node [below] at (6,0) {a copy of $\overline{J}_{a-b-1}$};
\end{tikzpicture}
\end{center}
\caption{$\iota_e'(a,b)$}
\label{fig:iota_e'}
\end{figure}
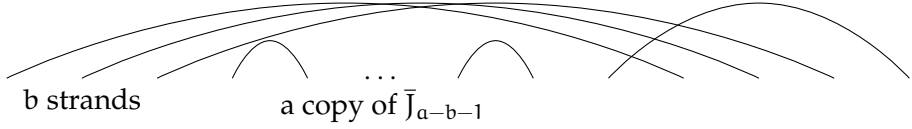

To obtain a symplectic essential rank condition of $(2r+1)$ in box $(i,j)$ we again need to do a case analysis on the parity of $i-j$:  

If $i-j$ is even use the fixed-point-free involution $\overline{J}_r \oplus \iota_o(i-2r,j-2r) \oplus \overline{J}_{r-(i+j)/2}$, where  $\iota_o(a,b)$ is given in Figure \ref{fig:iota_o}.  For example, $\iota_o(2,4)=361542$.   

If $i-j$ is odd use the fixed-point-free involution $\overline{J}_r \oplus \iota_o'(i-2r,j-2r) \oplus \overline{J}_{r-(i+j-1)/2}$, where  $\iota_o'(a,b)$ is described in Figure \ref{fig:iota_o'}.  For example, $21573846$ has one even symplectic essential box at $(3,4)$ with associated rank condition $2$ and one odd symplectic essential box at $(4,6)$ with associated rank condition $3$.  

\begin{figure}
\begin{center}
\begin{tikzpicture}
\draw plot [smooth, tension=1] coordinates {(1,0) (7,1) (13,0)};
\draw plot [smooth, tension=1] coordinates {(2,0) (8,1) (14,0)};
\draw plot [smooth, tension=1] coordinates {(3,0) (9,1) (15,0)};
\draw plot [smooth, tension=1] coordinates {(4,0) (5,.5) (6,0)};
\draw plot [smooth, tension=1] coordinates {(7,0) (7.5,.5) (8,0)};
\draw plot [smooth, tension=1] coordinates {(10,0) (10.5,.5) (11,0)};
\draw plot [smooth, tension=1] coordinates {(12,0) (14.5,1) (17,0)};
\draw plot [smooth, tension=1] coordinates {(5,0) (11.5,1) (16,0)};
\node [below] at (2,0) {$b-2$ strands};
\node at (9,0) {$\cdots$};
\node [below] at (9,0) {a copy of $\overline{J}_{a-b-1}$};
\end{tikzpicture}
\end{center}
\caption{$\iota_o(a,b)$}
\label{fig:iota_o}
\end{figure}
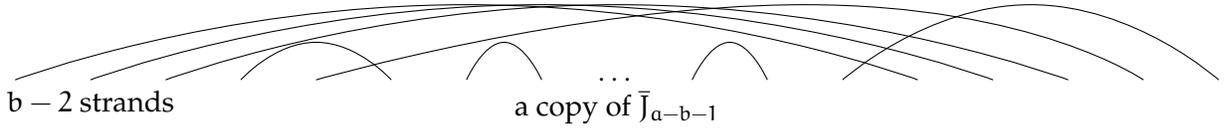

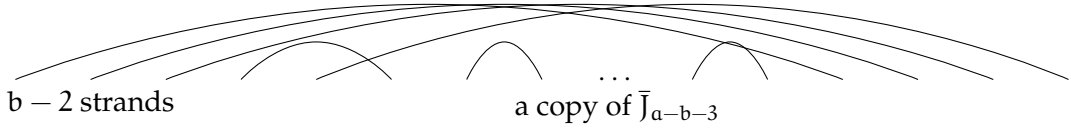
\begin{figure}
\begin{center}
\begin{tikzpicture}
\draw plot [smooth, tension=1] coordinates {(1,0) (6.5,1) (12,0)};
\draw plot [smooth, tension=1] coordinates {(2,0) (7.5,1) (13,0)};
\draw plot [smooth, tension=1] coordinates {(3,0) (8.5,1) (14,0)};
\draw plot [smooth, tension=1] coordinates {(7,0) (7.5,.5) (8,0)};
\draw plot [smooth, tension=1] coordinates {(10,0) (10.5,.5) (11,0)};
\node [below] at (2,0) {$b-2$ strands};
\node at (9,0) {$\cdots$};
\node [below] at (9,0) {a copy of $\overline{J}_{a-b-3}$};
\draw plot [smooth, tension=1] coordinates {(4,0) (5,.5) (6,0)};
\draw plot [smooth, tension=1] coordinates {(5,0) (10,1) (15,0)};
\end{tikzpicture}
\end{center}
\caption{$\iota_o'(a,b)$}
\label{fig:iota_o'}
\end{figure}

Before we prove Theorem \ref{thm:basicElts} we will need a lemma which shows that the even box symplectic essential rank condition for the elements of $\mathcal{A}_o$ is strictly necessary.  

\begin{lemma}\label{lemma:OddRanks}
If a fixed-point-free involution has a box in its symplectic diagram at location $(i,j)$ with an odd rank condition $2k+1$ then the rank condition at box $(i-1,i)$ must be at most $2k$.  
\end{lemma}

\emph{Proof}  The rank condition in cell $(i-1,i-1)$ must be even by antisymmetry of the matrices and must also be at most $2k+1$, hence it must be at most $2k$.  If the rank condition associated to cell $(i-1,i)$ is larger than $2k$ there must be a $1$ in column $i$ in the permutation matrix weakly above row $i-1$.  If that is the case then there must be a $1$ in the permutation matrix in row $i$ weakly to the left of column $i-1$, contradicting that the box $(i,j)$ was in the diagram of the fixed-point-free involution.  \qed

Note that Lemma \ref{lemma:OddRanks} implies that we cannot place a symplectic essential box with an odd rank condition on the immediate super-diagonal.  However, we can put a symplectic essential box with an odd rank condition anywhere else with an even symplectic essential box in the previous row on the immediate superdiagonal and no other symplectic essential boxes.  

We are now in a position to prove Theorem \ref{thm:basicElts}:

\emph{Proof of Theorem \ref{thm:basicElts}}  We will show each fixed point free involution is the greatest lower bound of some elements of $\mathcal{A}$.  Fix a fixed-point-free involution $\iota$.  

For each symplectic essential box at $(i,j)$ in the diagram of $\iota$ with even rank condition $r$, $\iota \le \iota'$ where $\iota'$ is the involution in $\mathcal{A}_e$ with exactly one symplectic essential box at $(i,j)$ of even rank condition $r$.   

For each symplectic essential box at $(i,j)$  in the diagram of $\iota$ with odd rank condition $r$, fix $\iota'$ where $\iota'$ is the involution in $\mathcal{A}_o$ with exactly two symplectic essential rank conditions, rank condition $r$  at $(i,j)$ and rank condition $r-1$ at $(i-1,i)$.  

By Lemma \ref{lemma:OddRanks}, $\iota \le \iota'$ for all of the $\iota'$ chosen above.  Since for each of the symplectic essential rank conditions for $\iota$ we have provided an $\iota'$ with that symplectic essential rank condition such that $\iota \le \iota'$, $\iota \le \text{glb}\{ \iota' \}$.  Further, as we have imposed no extra rank conditions not met by $\iota$, $\iota \ge \text{glb}\{ \iota' \}$.  \qed 

\section{Proof of Theorem \ref{thm: whichSchuberts}}

Now we to turn to the proof of Theorem \ref{thm: whichSchuberts}, which describes the degeneration of symplectic orbit closures.  Recall that we use a variant of the anti-diagonal term order in which we allow ties in the ordering of monomial terms such that the northwest-most determinant in a sum of determinants will be picked as the initial ``form" for the sum rather than choosing a single initial term.  This can be accomplished by weighting the columns of a matrix of variables $M$: assign weight $t^{\lceil j/2 \rceil -1}$ to the variable $m_{ij}$.  Then, we shall find initial forms by taking the limit as $t \to 0$.  

The equations for these orbit degenerations are part of a larger class of ideals whose generators are given in \cite{NWunions}; the equations for the orbits themselves remain unknown.  Data on the known equations is given in Section \ref{sec:data}.  In order to prove Theorem \ref{thm: whichSchuberts} we will first need some lemmas:


\begin{lemma}\label{lemma:SchemesRight}
If $w$ is a pair permutation for $\iota$, then $X^w_\circ$ and $(B_- \dom Y_\iota^\circ)$ intersect transversely in the reduced point $B_- \dom B_- w$.  
\end{lemma}

\begin{proof}
We need to show that $T_{B_-\dom B_- w} X^w \cap T_{B_-\dom B_- w} Y_\iota =\{ 0 \} \subseteq T_{B_-\dom B_- w}B.$  Phrasing this statement in terms of the corresponding Lie algebras, we must show that 
\[
(\b_- w^{-1} +w^{-1}\b_-) \cap (\b_- w^{-1} +w^{-1}\spn)\subseteq \b_- w^{-1},
\]
i.e. $(w\b_- w^{-1} +\b_-) \cap (w\b_- w^{-1} +\spn)\subseteq w\b_- w^{-1}$.  Hence it shall be sufficient to show  that $\b_- \cap \spn\subseteq w \b_- w^{-1}$.  

The zero entries in $w\b_-w^{-1}$ are those $(i,j)$ such that $w(i)>w(j)$.  Fix an entry $(i,j)$ of an arbitrary matrix $M \in \b_- \cap \spn$.  We shall show that if $w(i)>w(j)$ then $M_{ij}=0$.  First notice that if $i<j$ then $M_{ij}=0$ as $M \in \b_-$.  Now examine the case that $i>j$.    
\[
\spn=\{ [A_{kl}]_{1\le k,l \le n}: A_{kl}\text{ is a } 2 \times 2\text{ matrix s.t. } A_{kl}=J_2A_{lk}^TJ_2\}
\]
where $J_2$ is the $2 \times 2$ matrix $\left(\begin{smallmatrix}0 & 1 \\-1 & 0\end{smallmatrix}\right)$.  This is equivalent to insisting that matrices $M \in \spn$ have entries $M_{ij}$ such that 
\[
M_{ij}=\begin{cases}
M_{j+1i+1} & i \text{ odd  and } j \text{ odd}\\
-M_{ji} & i \text{ odd  and } j \text{ even}\\
-M_{ji} & i \text{ even  and } j \text{ odd}\\
M_{j-1i-1} & i \text{ even  and } j \text{ even}
\end{cases}
\]
This means that if $i>j$, $M_{ij}=M_{kl}$ for $k<l$ since $M \in \spn$ and as $M \in b_-$ this implies that $M_{ij}=0$.  \qed

\begin{lemma}\label{lemma:dimsRight}
$\dim Y_\iota = (2n)^2-\dim X^w$
\end{lemma}

\emph{Proof} Let $c$ be the number of countryside involutions contained in $\iota$ and let $r$ be the number of rainbow involutions contained in $\iota$  Notice that $\dim Y_\iota=(8n^2+n-l(\iota))/2=(8n^2+n-(n+2c+4r))/2$ and that $\dim X^w=l(w)=c+2r$.  Then we can observe that $\dim X^w=(2n)^2-\dim Y_\iota$.  \qed 

\begin{lemma}\label{lemma:PermsRight}
The elements $w \in S_{2n}$ of minimal length such that $w\overline{J}w^{-1}= \iota$ are the pair permutations for $\iota$.  
\end{lemma}

\emph{Proof}  We will show first that $wJw^{-1}= \iota$ by instead showing the equivalent $w^{-1}Jw= \iota$.  The best way to see this is to follow a pair of strands used to make a circle in the diagram description of pair permutations.  There are two very similar cases, dependent on whether $w(i)$ is odd or even for a fixed, general $i$.  See Figure \ref{fig:permsRight}, which is labeled for the case that $w(i)$ is odd, though the same picture with different labels applies for the case that $w(i)$ is even.  If $w^(i)$ is odd then $\overline{J}_nw(i)=w(i)+1$ and the definition of a pair permutation requires that $w(i+1)=\iota(i)$.  Similarly, if $w(i)$ is even then $\overline{J}_nw(i)=w(i)-1$ and the definition of a pair permutation requires that $w(i-1)=\iota(i)$.  The pair permutations are of minimal length because we minimize the number of wire crossings.  \qed

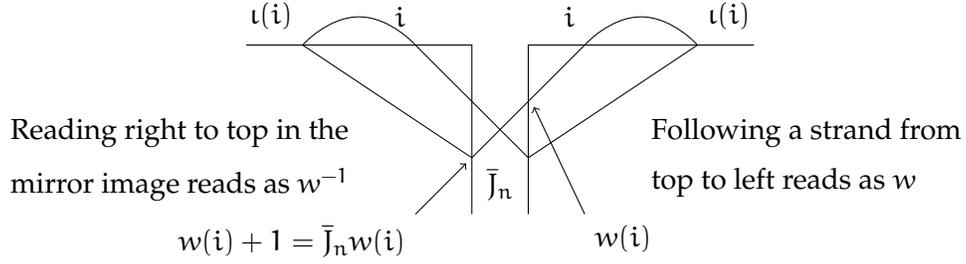
\begin{figure}
\begin{center}
\begin{tikzpicture}[x=.75cm, y=.75cm]
\draw plot [smooth, tension=1] coordinates {(1,3) (2,3.5) (3,3)};
\draw plot [smooth, tension=1] coordinates {(6,3) (7,3.5) (8,3)}; 
\draw(0,3)--(4,3)--(4,0);
\draw (5,0)--(5,3)--(9,3);
\draw (3,3)--(4,2)--(5,1)--(8,3);
\draw (1,3)--(4,1)--(5,2)--(6,3);
\node[left] at (6,3.5) {$i$};
\node [right] at (8,3.5) {$\iota(i)$};
\node[left] at (1,3.5) {$\iota(i)$};
\node[left] at (3,3.5) {$i$};
\node[above right] at (7,1) {Following a strand from};
\node[below right] at (7,1) {top to left reads as $w$};
\node[above left] at (2,1) {Reading right to top in the};
\node[below left] at (2,1) {mirror image reads as $w^{-1}$};
\node [below] at (4.5,1) {$\overline{J}_n$};
\node [below left] at (3,0) {$w(i)+1=\overline{J}_nw(i)$};
\draw [->] (3,0)--(3.9,.9);
\draw [->] (6,0)--(5.1,1.9);
\node [below right] at (6,0) {$w(i)$};
\end{tikzpicture}
\caption{The proof of Lemma \ref{lemma:PermsRight}.}
\label{fig:permsRight}
\end{center}
\end{figure}

\emph{Proof of Theorem \ref{thm: whichSchuberts}}  By \cite{Brion}, the $Y_\iota$ degenerate to a union of Schubert varieties and it is enough to show that $(B_- \dom Y_\iota) \cap X^w \ne \varnothing$, where $X^w=B_- \dom B_- w B_+$.  By Borel's fixed point theorem, $(B_- \dom Y_\iota) \cap X^w \ne \varnothing$ if and only if $(B_- \dom Y_\iota \cap X^w)^{T_{Sp_n}} \ne \varnothing$.   Here $T_{Sp_n}$ is the torus for $Sp_n$ , $T$ is the standard torus in $GL_{2n}\C$ and $[1,w]=\{v \in S_{2n}: 1 \le v \le w\}$.  But 
\begin{align*}
(B_- \dom Y_\iota \cap X^w)^{T_{Sp_n}} &= (B_- \dom Y_\iota)^{T_{Sp_n}} \cap (X^w)^T\\
&= (B_- \dom Y_\iota)^{T_{Sp_n}} \cap [1,w]\\
&= \cup_{\iota' \le \iota}(B_- \dom Y_{\iota'}^\circ)^{T_{Sp_n}} \cap [1,w]\\
&= \{ v: vJv^{-1}= \iota\text{ and } v \text{ is of minimum length}\}.
\end{align*}
Lemmas \ref{lemma:SchemesRight}, \ref{lemma:dimsRight} and \ref{lemma:PermsRight} now complete the proof.  
\end{proof}

\section{Data and Remaining Open Questions}\label{sec:data}

In an ideal world, for each element $\iota' \in \mathcal{A}$ we would have a Gr\"obner basis for $I(Y_{\iota'})$ under some antidiagonal term order.  Then for each fixed point free involution $\iota$ we could impose all of the equations for each $\iota' \in \mathcal{A}$ such that $Y_{\iota'} \supseteq Y_\iota$ and this would be a Gr\"obner basis for $Y_{\iota'}$ by results from \cite{KnutsonFrob}.  Sadly, we do not have a Gr\"obner basis for elements of $\mathcal{A}$ in general, but we summarize what we know of the equations in this section.   Unusually, in this case the equations for the degenerations are known, see \cite{NWunions}, even though the equations for the original schemes are unknown.  

The \defi{pfaffian} of an antisymmetric matrix is a choice of the square root of its determinant, which is a square.  Since we shall use it as a generator for an ideal we shall not concern ourselves with which choice of the square root.  

\begin{proposition}
Assume $\iota$ has one symplectic essential box at $(2r-1, 2r)$ and that that box has associated rank condition $2r-2$.  Then the reduced variety $Y_\iota$ is defined by the pfaffians of rows and columns $\{1, \ldots, 2r\}$ of $MJM^T$ where $M$ is a $2n \times 2n$ antisymmetric matrix of variables.  
\end{proposition}

\emph{Proof} Expand the symplectic diagram to the full Rothe diagram using antisymmetry of the matrices and apply Fulton's Theorem to $MJM^T$ to get one generator--the determinant of the northwest $2r \times 2r$ submatrix.  Then recall that the determinant of an antisymmetric matrix, which is the only Fulton generator in this case, is a square with square root the pfaffian of that matrix.   \qed

\begin{lemma}
Providing the equations for the reduced varieties associated to the $Y_\iota$ requires calculating the radical of a variety with ideal given by determinants of $MJM^T$.  
\end{lemma}

\emph{Proof}  We wish to describe the variety that corresponds to the orbit $Y_\iota$.  We do this by applying the map given by \cite{RichardsonSpringer1990} $M \mapsto MJM^T$ to a matrix of variables $M$.  Then, we apply Fulton's Theorem (Theorem \ref{thm:Fulton1992}) and take the radical to get the correct description for the reduced scheme.  \qed

This has proven computationally infeasible even in fairly small examples--further insight is needed.  We have had some success with computing a few examples using Macaulay2 \cite{M2}, summarized in Table \ref{table:eqns}.  In Table \ref{table:eqns}, $\det((MJM^T)_{R,C})$ is the determinant of rows $R$ and columns $C$ of $MJM^T$, where $M$ is a matrix of variables.  Similarly, $\pf((MJM^T)_{R,C})$ is the pfaffians of rows $R$ and columns $C$ of $MJM^T$, where $M$ is a square matrix of variables of size the number of elements that $\iota$ is permuting.  Note that the equations found in Table \ref{table:eqns} are mostly not a Gr\"obner basis for the ideals they generate.  

\begin{table}[htdp]
\caption{Generators for some $I(Y_\iota)$}
\begin{center}
\begin{tabular}{|c|c|}
\hline
$\iota$ & Ideal Generators \\ \hline 
$4321\oplus \overline{J}_{n-2}$ & $\pf((MJM^T)_{\{1,2\}\{1,2\}})$, $\pf((MJM^T)_{\{1,3\},\{1,3\}})$ \\ \hline
$216543$ & $\pf((MJM^T)_{\{1,2,3,4\},\{1,2,3,4\}})$, $\pf((MJM^T)_{\{1,2,3,5\},\{1,2,3,5\}})$ \\ \hline
$351624$ & $\pf((MJM^T)_{\{1,2\},\{1,2\}})$, $\pf((MJM^T)_{\{1,2,3,4\},\{1,2,3,4\}})$\\ \hline
\end{tabular}
\end{center}
\label{table:eqns}
\end{table}

\bibliography{SpnOrbitsRefs}
\bibliographystyle{amsalpha}

\end{document}